\documentclass[11pt]{article}

\usepackage{geometry}
\usepackage{authblk}
\usepackage{lmodern}
\usepackage{amsmath}\usepackage{amscd}
\usepackage{amssymb}
\usepackage{theorem}
\usepackage{xypic}
\usepackage[all,v2]{xy}
\xyoption{2cell}
\UseAllTwocells
\usepackage{enumitem}
\usepackage{rotating}
\usepackage{mathrsfs}
\usepackage{stmaryrd}
\usepackage{mathdots}
\usepackage{hyperref}
\usepackage{caption}
\usepackage[normalem]{ulem}
\usepackage{xcolor}

\usepackage{subcaption}
\usepackage{tikz}
\usepackage{tikz-cd}
\usepackage{fixltx2e}

\newenvironment{items}
{\begin{enumerate}[topsep=3pt, itemsep=3pt, parsep=0pt, label=(\roman*)]}
{\end{enumerate}}

\renewcommand{\tilde}{\widetilde}
\newcommand{\dual}[1]{#1^{\vee}}

\newcommand{\A}{{\mathscr A}}

\renewcommand{\H}{{\mathscr H}}

\renewcommand{\L}{{\mathscr L}}

\newcommand{\M}{{\mathscr M}}
\newcommand{\Y}{{\mathscr Y}}
\newcommand{\N}{{\mathscr N}}

\renewcommand{\P}{{\mathscr P}}

\newcommand{\rr}{{\mathbb R}}

\renewcommand{\O}{{\mathscr O}}

\newcommand{\cC}{{\mathscr C}}

\newcommand{\st}{\mathrel{\mid}}

\newcommand{\ev}{\mathop{\rm ev}\nolimits}

\newcommand{\Mor}{\mathop{\rm Mor}\nolimits}

\newcommand{\injectlim}{\mathop{\lim\limits_{\textstyle\longrightarrow}}}

\newcommand{\longiso}{\stackrel{\textstyle\sim}{\longrightarrow}}

\newcommand{\doublearrowstack}[2]%
                      {{{{\scriptstyle#1}\atop{\textstyle\longrightarrow}}\atop{{\textstyle\longrightarrow}\atop{\scriptstyle#2}}}}
\newcommand{\rightleftarrowstack}[2]%
                      {{{{\scriptstyle#1}\atop{\textstyle\longrightarrow}}\atop{{\textstyle\longleftarrow}\atop{\scriptstyle#2}}}}
\newcommand{\leftrightarrowstack}[2]%
                      {{{{\scriptstyle#1}\atop{\textstyle\longleftarrow}}\atop{{\textstyle\longrightarrow}\atop{\scriptstyle#2}}}}

\newtheorem{thm}{Theorem}[section]
\newtheorem{cor}[thm]{Corollary}
\newtheorem{lem}[thm]{Lemma}
\newtheorem{prop}[thm]{Proposition}

\theorembodyfont{\upshape}
\newtheorem{defn}[thm]{Definition}

\newtheorem{rmk}[thm]{Remark}

\newenvironment{pf}{\begin{trivlist}\item[]{\sc Proof.}}%
            {\nolinebreak $\Box$ \end{trivlist}}

            {\nolinebreak $\Box$ \end{trivlist}}

\newenvironment{pfrecap}{\begin{trivlist}\item[]{\sc Proof of Theorem~\ref{recap}.}}%
            {\nolinebreak $\Box$ \end{trivlist}}

\newenvironment{pfStep1}{\begin{trivlist}\item[]{\sc Proof of Step~1.}}%
            {\nolinebreak $\Box$ \end{trivlist}}
            
\newenvironment{pfStep2}{\begin{trivlist}\item[]{\sc Proof of Step~2.}}%
            {\nolinebreak $\Box$ \end{trivlist}}

\newenvironment{pfqiso}{\begin{trivlist}\item[]{\sc Proof of Theorem~\ref{thm:WeakEqImplyQuasiIso}.}}%
            {\nolinebreak $\Box$ \end{trivlist}}            

\newenvironment{proof}{\begin{trivlist}\item[]{\sc Proof.}}%
            {\nolinebreak $\Box$ \end{trivlist}}

\newcommand{\del}{\partial}
\newcommand{\longto}{\longrightarrow}
\DeclareMathOperator\id{id}

\newcommand{\im}{\mathop{\rm im}\nolimits}

\newcommand{\pr}{\mathop{\rm pr}\nolimits}

\newcommand{\Sym}{\mathop{\rm Sym}\nolimits}

\newcommand{\Hom}{\mathop{\rm Hom}\nolimits}

\newcommand{\contract}{{\,\lrcorner\,}}

\newcommand{\noprint}[1]{}

\newcommand{\eetale}{a local diffeomorphism}

\newcommand{\Ho}{\mathrm{Ho}}

\newcommand{\catL}{\mathscr{L}}
\newcommand{\germL}{\catL_{\text{germ}}}

\title{On the structure of  \'etale fibrations of $L_\infty$-bundles }

\author{Kai Behrend,
Hsuan-Yi Liao
and Ping Xu}

\date{{\it Dedicated to the memory of Bumsig Kim}}

\newcommand{\Addresses}{{
  \bigskip
  \footnotesize

  Kai Behrend, \textsc{Department of Mathematics,  University of British Columbia}\par\nopagebreak
  \textit{E-mail address}: \texttt{behrend@math.ubc.ca}

  \medskip

  Hsuan-Yi Liao, \textsc{Department of Mathematics, 
National Tsing Hua University}\par\nopagebreak
  \textit{E-mail address}: \texttt{hyliao@math.nthu.edu.tw}

  \medskip

  Ping Xu, \textsc{Department of Mathematics, Pennsylvania State University}\par\nopagebreak
  \textit{E-mail address}: \texttt{ping@math.psu.edu}

}}

\begin{document}
\sloppy

\maketitle

\begin{abstract}
{We prove that an \'etale fibration between $L_\infty$-bundles admits local sections composed of several elementary morphisms of particularly simple and accessible type. As applications, we establish an inverse function theorem for $L_\infty$-bundles and provide an elementary proof that every weak equivalence of $L_\infty$-bundles induces a quasi-isomorphism of the differential graded algebras of global functions. Furthermore, we apply this inverse function theorem to show that the homotopy category of $L_\infty$-bundles admits a simple description in terms of homotopy classes of morphisms, when $L_\infty$-bundles are restricted to their germs around their classical loci.
} 

\medskip
\noindent {2020 Mathematics Subject Classification {58A50}.}
\end{abstract}

\let\thefootnote\relax\footnotetext{Research partially supported by NSF grants DMS-2001599 and DMS-2302447, KIAS Individual Grant MG072801 and MOST/NSTC Grant
110-2115-M-007-001-MY2 and 112-2115-M-007-016-MY3.}

\tableofcontents

\section*{Introduction}\addcontentsline{toc}{section}{Introduction}

The present paper is a contribution to the foundations of derived geometry in the $C^\infty$-context.  The setting is the theory of $L_\infty$-bundles \cite{BLX1}. The paper concerns the structure of morphisms of $L_\infty$-bundles.  We prove an inverse function theorem, and we provide applications to the structure of the homotopy category of $L_\infty$-bundles.

An {\em $L_\infty$-bundle} is a triple $\M=(M,L,\lambda)$, where $M$ is a $C^\infty$-manifold, $L=L^1\oplus\ldots\oplus L^n$ is a finite-dimensional graded vector bundle of positive amplitude over $M$, and $\lambda=(\lambda_k)_{0\leq k< n}$ is a sequence of bundle maps $\lambda_k:\Sym^k L \to L$ of degree $1$, making each fiber $L|_P$, where $P\in M$, a curved $L_\infty[1]$-algebra.  
The zeroth operation $\lambda_0$ is a global section of $L^1$ (called the {\em curvature}). 
We will refer to the points of $M$ at which $\lambda_0$ vanishes as the {\em classical points}, and to the set of all classical points as the {\em classical locus} or the {\em Maurer--Cartan locus} of $\M$, denoted $\pi_0(\M)$.

A remark on terminology: having all operations of the same degree entails significant simplifications when dealing with signs. Some may regard the use of the term $L_\infty$-bundle, rather than $L_\infty[1]$-bundle, as an abuse of notation, but there is precedent for this usage in the literature, e.g.~\cite{ezra}.

A {\em morphism }of $L_\infty$-bundles $(f,\phi):(M,L,\lambda)\to
(M',L',\lambda')$ consists of a differentiable map  of manifolds $f:M\to M'$, and 
a sequence of bundle maps 
$$
\phi_k:\Sym^k L \to L', \qquad k \geq 1,
$$ 
of degree zero over $f$. It is required that for every $P\in M$, the induced
sequence of maps $\phi|_P$ defines a morphism of curved $L_\infty[1]$-algebras $(L,\lambda)|_P\to (L',\lambda')|_{f(P)}$. 
There are two particularly important classes of morphisms in the category of $L_\infty$-bundles --- {\em fibrations} and {\em weak equivalences}.  
A {\em fibration} is a morphism of $L_\infty$-bundles $(f,\phi):(M,L,\lambda)\to
(M',L',\lambda')$ such that $f:M\to M'$ is a submersion, and $\phi_1:L\to f^\ast L'$ is a degree-wise surjective morphism of  graded vector bundles over $M$. 
A {\em weak equivalence} is a morphism of $L_\infty$-bundles which is \'etale and induces a bijection on classical loci. (A morphism is {\em \'etale} if it induces a quasi-isomorphism on tangent complexes at all points of the
 classical locus.)  
It is proved in \cite{BLX1} that $L_\infty$-bundles form a category of fibrant objects in the sense of Brown \cite{MR341469}. This category contains both the category of $C^\infty$-manifolds, and the category of finite dimensional, positively graded $L_\infty[1]$-algebras  as  full subcategories. This feature of the category of $L_\infty$-bundles suggests its importance in derived differential geometry \cite{MR2641940,carchedi2012homological, MR3121621, MR3221297, borisov2011simplicial,  2018arXiv180407622P, Nuiten, MR4722106,2019arXiv190506195C}.

In the present paper, we investigate the structure of \'etale and trivial fibrations between $L_\infty$-bundles. By {\em trivial fibration}, we mean a fibration which is also a weak equivalence. For instance, if $U$ is an open neighborhood of the classical locus of $\M$ in the base manifold $M$, then the inclusion map $\M|_U \hookrightarrow \M$ is a trivial fibration. Recall that every fibration is equal to the composition of an isomorphism followed by a {\em linear fibration}  (Proposition~\ref{prop:IsoToLinearMor}). Here, a morphism $(f,\phi)$ is said to be {\em linear} if $\phi_k$ vanish for all $k >1$. 
Therefore, we can focus our attention on linear fibrations. 

A main technique in studying $L_\infty$-bundles is the {\em (homotopy) transfer theorem} (Proposition~\ref{transfertheorem}). When applying the transfer theorem to an $L_\infty$-bundle $\M$, we obtain a weakly equivalent $L_\infty$-bundle $\M'$ over the same base manifold, together with an inclusion morphism $\iota:\M' \to \M$ and a projection morphism $\pi:\M \to \M'$, such that the composition $\pi \bullet \iota$ is the identity morphism, where both $\iota$ and $\pi$ are weak equivalences.
A {\em transfer embedding} $(U,H,\mu) \to (M,L,\lambda)$ is  the composition of an inclusion morphism induced from the transfer theorem and a trivial fibration of the type $\M|_U \hookrightarrow \M$, where $U$ is an open neighborhood of the classical locus of $\M$. 

A particularly simple kind of linear $L_\infty$-morphisms is given as follows.  Let $\M=(M,L,\lambda)$ be an $L_\infty$-bundle.  Let $Y\subset M$ be a submanifold, and $E^1\subset L^1|_Y$ a subbundle, such that the curvature $\lambda_0|_Y:Y\to L^1|_Y$ factors through $E^1$. Then $\Y=(Y,E^1 \oplus L^{\geq2}|_Y)$ is an $L_\infty$-bundle in a canonical way, and the inclusion $\Y\to \M$ is a linear morphism.  Let us call such embeddings {\em simple embeddings}.

Our main technical result can be summarized by saying that linear \'etale fibrations admit particularly ``nice'' sections, at least after replacing the target with a locally isomorphic one:

\begin{trivlist}
\item {\bf Theorem A} (Theorem~\ref{recap}) Let $f:\M\to\N$ be a linear  \'etale  fibration.  Then there exists a commutative diagram of $L_\infty$-bundles 
$$\begin{tikzcd}
\M'\arrow[d,"\iota"'] & \N'\arrow[d,"u"]\arrow[l,"s"']\\
\M\arrow[r,"f"] & \N\rlap{\,,}
\end{tikzcd}$$
where $u:\N'\to \N$ is a local isomorphism, $s:\N'\to\M'$ is a simple embedding, and $\iota:\M'\to \M$ is a finite composition of  transfer embeddings. Moreover, $u$ and $s$ are linear.   

Although $\iota$, the composition of transfer embeddings, is not linear, the composition $f\circ\iota:\M'\to \N$ is linear, and induces isomorphisms of bundles in degrees two and higher. 
\end{trivlist}

We can view this theorem as exhibiting how $\M$ is built up from $\N$ (or $\N'$)  in several steps. In this sense, we consider Theorem~A  a structure theorem for \'etale fibrations.

\subsubsection*{Inverse function theorem}

When we apply Theorem~A to the case of trivial fibrations, adding the assumption that $f$ induces a bijection on classical loci, we can strengthen the result, and replace the local isomorphism $u$ by the inclusion $\N|_U\to\N$ of an open neighborhood of the classical locus in $\N$. We obtain:

\begin{trivlist}
\item {\bf Theorem B} (Corollary~\ref{cor:LocalSec}) {\it   Every trivial fibration of $L_\infty$-bundles admits a section, at least after restricting to an open neighborhood of the classical locus of the target.}
\end{trivlist}

An important application is to the structure of the homotopy category of $L_\infty$-bundles. We define  the {\em category of germs of $L_\infty$-bundles} by declaring 
$$\Mor_{\text{germ}}(\M,\N)=\injectlim_{U}\Mor_{\text{$L_\infty$-bundles}}(\M|_U,\N)\,,$$
where the limit is taken over all open neighborhoods of the classical locus of $\M$.  This new category inherits the structure of category of fibrant objects, and still has the same homotopy category. It has the additional property, that every trivial fibration admits a section.  Such categories of fibrant objects have a particularly simple description of their homotopy categories:  the morphisms are simply homotopy classes of morphisms
 (Corollary~\ref{cor:htpCat}).   Here two morphism germs $f,g:\M\to \N$ are homotopic if there exists a morphism germ $h:\M\to\P\N$ from $\M$ to a path space object $\P\N$  of $\N$, which gives back $f$ and $g$ upon composing with the two evaluation maps $\P\N\to\N$. (The existence of $\P\N$ and the two evaluation maps is guaranteed by the axioms of category of fibrant objects.)

After the posting of the first e-print version \cite{2020arXiv200601376B} on
 arXiv.org, we learned that  a similar version of the inversion function theorem for $L_\infty$-bundles (called  $L_\infty$ spaces by the authors)  
was also  obtained by Amorim-Tu using a different method \cite{MR4542391}.

If $L$ is entirely concentrated in degree 1 (so $n=1$), we speak of {\em quasi-smooth }$L_\infty$-bundles. The notion of quasi-isomorphism of quasi-smooth $L_\infty$-bundles was studied by Mather  in the 1960s (see the notion of $\mathscr{C}$-equivalence, \cite[Proposition 2.3]{MatherIII}).   Mather restricts to the equidimensional case.  He develops interesting criteria for how to detect quasi-isomorphisms of germs by considering jets.  It might be interesting to develop such criteria for the more general context of our paper, but this is left for future research.

\subsubsection*{The associated differential graded algebra}

By standard facts, we can associate to an $L_\infty$-bundle $(M,L,\lambda)$:
\begin{items}
\item a differential graded manifold whose structure sheaf is the sheaf of commutative differential graded algebras $\A = \Sym_{\O_M}L^\vee$, equipped with a derivation $Q:\A \to \A$ of degree 1, induced by $(\lambda_k)$, satisfying $Q^2 = \frac{1}{2}[Q,Q] = 0$;
\item the commutative differential graded algebra $\Gamma(M,\A)$ of global sections.
\end{items}

As another application of Theorem~A, we give an elementary proof of the following

\begin{trivlist}
\item {\bf Theorem C} (Theorem~\ref{thm:WeakEqImplyQuasiIso}) {\it  Any weak equivalence  of $L_\infty$-bundles induces a quasi-isomorphism of the associated differential graded algebras of global sections.}
\end{trivlist}

The importance of this theorem lies in the fact that it establishes that the differential graded algebra of global sections is an invariant of an $L_\infty$-bundle:  Weakly equivalent $L_\infty$-bundles have quasi-isomorphic global section algebras.

The converse, namely that every morphism of $L_\infty$-bundles, which induces a quasi-isomorphism on global section differential graded algebras is a weak equivalence,  has been proved using  techniques from the theory of dg $C^\infty$-rings, in particular the existence of the cotangent complex.  
See \cite{Nuiten, 2018arXiv180407622P, 2023arXiv230311140C,2023arXiv230408671S,2023arXiv230312699T}.

\subsection*{Notations and conventions}

Differentiable means $C^\infty$. Manifold means differentiable manifold, which includes second countable and Hausdorff as part of the definition.  Hence manifolds admit partitions of unity, which implies that vector bundles admit connections, and fiberwise surjective homomorphisms of vector bundles admit sections.

For any graded vector bundle $E$ over a manifold $M$, we sometimes use the same symbol $E$ to denote the sheaf of its sections over $M$, by abuse of notation. 
In particular, for a graded vector bundle $L$ over $M$, by
$\Sym_{\O_M} L^\vee$, or simply $\Sym L^\vee$, we denote the sheaf, over $M$,
 of sections of the graded vector bundle $\Sym L^\vee$, i.e. the sheaf of
fiberwise polynomial functions on $L$.

The bundle maps are over the identity maps unless otherwise stated. 

For graded vector bundles $E$ and $F$ over a manifold $M$, we denote by $\Sym_M^n (E,F)$ the space of vector bundle maps from $\Sym E$ to $F$ of degree $n$. 

The notation $| \cdot |$ denotes the degree of an element. When we use this notation, we assume the input is a homogeneous element.

\section*{Acknowledgments}

We would like to thank several institutions for their hospitality, which made the completion of this project possible: Pennsylvania State University (Liao); the University of British Columbia (Liao and Xu); the National Center for Theoretical Science (Liao); KIAS (Xu); the Institut des Hautes \'{E}tudes Scientifiques (Liao); and the Institut Henri Poincar\'{e} (Behrend, Liao, and Xu).

We are also grateful to Ruggero Bandiera, Damien Broka, David Carchedi, Alberto Cattaneo, David Favero, Domenico Fiorenza, Ezra Getzler, Owen Gwilliam, Vladimir Hinich, Bumsig Kim, Wille Liu, Marco Manetti, Raja Mehta, Pavel Mnev, Joost Nuiten, Byungdo Park, Jonathan Pridham, Dima Roytenberg, Pavol \v{S}evera, Mathieu Sti\'enon, and Ted Voronov for fruitful discussions and valuable comments.

\section{Preliminaries}

We briefly summarize some important facts about $L_\infty$-bundles in this section. See \cite{BLX1} for more details.  

\subsubsection*{$L_\infty$-bundles}

Recall that a morphism of $L_\infty$-bundles $(f,\phi):(M,L,\lambda)\to
(M',L',\lambda')$ consists of a differentiable map of manifolds $f:M\to M'$, and $\phi \in \Sym_M^0(L,f^\ast L')$ such that $\phi \circ \lambda = \lambda' \bullet \phi$ --- see \cite[Section~1.2]{BLX1} for the precise definitions of $\circ$ and $\bullet$. A morphism $(f,\phi)$ is said to be linear if $\phi_k = 0$ for all $k \geq 2$. 
It is called a fibration if $f$ is a submersion, and $\phi_1:L \to f^\ast L'$ is surjective. The following proposition was proved in \cite{BLX1}.

\begin{prop}[\text{\cite[Lemma 3.24]{BLX1}}]
\label{prop:IsoToLinearMor}
Every fibration is equal to 
the composition of a linear fibration with an isomorphism. 
\end{prop}

A morphism $\M \to \M'$ of $L_\infty$-bundles is \'{e}tale at a classical point $P$ of $\M$ if it induces a quasi-isomorphism of tangent complexes at $P$. We say a morphism $\M \to \M'$ is \'{e}tale if it is \'{e}tale at every classical point of $\M$. A morphism $\M \to \M'$ is a weak equivalence if it is \'{e}tale and induces a bijection of classical loci. 
The main theorem in \cite{BLX1} is the following

\begin{thm}[\cite{BLX1}]\label{thm:main}
The category of $L_\infty$-bundles is a category of fibrant objects.
\end{thm}

\subsubsection*{Homotopy transfer and weak equivalence}

An important tool in the present paper is the homotopy transfer theorem. Let $(L,\delta)$ be a complex of vector bundles (of finite dimension, concentrated in finitely many positive degrees) over a manifold $M$. Assume that $L$ is endowed with a finite descending filtration 
\begin{equation}\label{eq:GeneralFiltation}
L=F_0 \supset F_1 \supset \ldots
\end{equation}
which is compatible with the grading and the differential $\delta$. 

\begin{rmk}\label{varfil}
There is a {\em natural filtration }of $L$ given by 
\begin{equation}\label{eq:NaturalFiltration}
F_k =\bigoplus_{i\geq k}L^i\,.
\end{equation}

Every $L_\infty[1]$-algebra structure on $L$ is {\em filtered} (i.e.\ it increases the filtration degree by $1$) with respect to the natural filtration~\eqref{eq:NaturalFiltration}.

In the proof of Lemma~\ref{firstcase}, we need the following variation of the natural
filtration $F$ at level $n$.  We define 
$$\tilde F_k=F_k\,,\qquad\text{for all $k\not=n$}\,,$$
and
$$\tilde F_n=\bigoplus_{i\geq n+1}L^{i}\,.$$
It is easy to check that 
$$
\lambda_m: \tilde{F}_{k_1} \times \ldots \times \tilde{F}_{k_m} \to \tilde{F}_{k_1 + \ldots + k_m +1}, \qquad \forall \, m \geq 2,
$$
and 
$$
\lambda_1:\tilde{F}_k \to \tilde{F}_{k+1}, \qquad \forall\, k \neq n-1.
$$
\end{rmk}

Let $\eta:L \to L$ be a vector bundle map of degree $-1$ which is compatible with the filtration $F$ and satisfies the following two equations:
$$\eta^2=0\,,\qquad \eta\delta\eta=\eta\,.$$
(We call $\eta$ a {\em contraction} of $\delta$.) 
Under these hypotheses, $\delta\eta$ and $\eta\delta$ are projection operators, and so is $[\delta,\eta]$.  We define $$H=\ker[\delta,\eta]=\ker(\delta\eta)\cap\ker(\eta\delta)=\im(\id_L-[\delta,\eta])\,$$ which is a graded vector subbundle of $L$. Let us write $\iota:H\to L$ for the inclusion, and $\pi:L\to H$ for the projection. We have 
$$\iota \pi=\id_L-[\delta,\eta]\,, \quad\text{and}\quad \pi\iota=\id_H\,.$$ 
We write $\delta$ also for the induced differential on $H$. Then $\iota:(H,\delta)\to (L,\delta)$ is a homotopy equivalence with homotopy inverse $\pi$. 

\begin{prop}[Homotopy Transfer Theorem]\label{transfertheorem}
Let $\lambda = (\lambda_k)_{k \geq 0}$ be a curved $L_\infty[1]$-structure on $(L,\delta)$. If $\lambda$ is filtered, then there is a unique  $\phi\in \Sym_M^0(H,L)$ satisfying the equation
\begin{equation}\label{recu}
\phi=\iota-\eta \lambda\bullet\phi\,.
\end{equation}
Setting
$$\mu=\pi\lambda\bullet\phi\in \Sym_M^1(H,H)$$ defines a curved $L_\infty[1]$-structure on $(H,\delta)$, such that $(\id,\phi)$ is a morphism of $L_\infty$-bundles from $(M,H,\delta+\mu)$ to $(M,L,\delta+\lambda)$. 

Furthermore, there exists a morphism of $L_\infty$-bundles $(\id, \tilde\pi):(M,L, \delta + \lambda) \to (M,H, \delta + \nu)$ satisfying the equations $\tilde\pi \bullet \phi = \id_H$ and $\tilde\pi_1 = \pi - \tilde\pi_1 \lambda_1 \eta$.
\end{prop}

\begin{rmk}
The formulas for $\phi$, $\mu$, and $\tilde{\pi}$ can be obtained by the
 method in \cite{MR3276839} (see also \cite{2020arXiv201214812B,BLX1, ezra,MR4693987}). In the cases of $\phi$ and $\mu$, explicit formulas can also be derived by solving
 \eqref{recu} recursively, resulting in expressions given as sums over trees. Since the filtration~\eqref{eq:GeneralFiltation} is assumed to have finite length, the operations $\phi$, $\mu$, and $\tilde{\pi}$ are finite sums of compositions of smooth operations. Hence, they are smooth.

\end{rmk}

\begin{prop}\label{prop:TransInclWkEq}
Let $(\id,\phi):(M,H,\delta+ \mu) \to (M,L,\delta+\lambda)$ be the inclusion morphism of $L_\infty$-bundles as in Proposition~\ref{transfertheorem}.
 Then $(\id,\phi)$ is a weak equivalence of $L_\infty$-bundles.
\end{prop}
\begin{pf}
Since $\tilde\pi_1\circ\phi_1=\id$, we have that $\phi_1:H^1\to L^1$ is an isomorphism onto a subbundle of $L^1$. Thus $\mu_0(P)=0$ is equivalent to $\lambda_0(P)=0$ for any $P \in M$.  In other words, the Maurer-Cartan loci of $(M, H,\delta+\mu)$ and $(M, L,\delta + \lambda)$ are equal.

Suppose $P$ is a Maurer-Cartan point.  Then $(H|_P,\delta+\mu_1)$ and $(L|_P,\delta+\lambda_1)$ are complexes of vector spaces, and $\phi_1$ is a morphism of complexes.  In fact, $$\phi_1: (H|_P,\delta+\mu_1)\to (L|_P,\delta+\lambda_1)$$ is a quasi-isomorphism. 
To prove this, considering the spectral sequence induced by the given filtration \eqref{eq:GeneralFiltation}, one can reduce the proof to the fact that 
$$\iota:(H|_P,\delta)\longrightarrow (L|_P,\delta)$$ 
is a quasi-isomorphism which is true.  Consequently, the vertical maps 
$$
\xymatrix{
TM|_P \rto \dto^\id &  \dto^{\phi_1|_P}   H^1|_P \rto^-{\delta + \mu_1} & H^2|_P \rto^-{\delta + \mu_1} \dto^{\phi_1|_P} & \ldots \\
TM|_P \rto &   L^1|_P \rto_-{\delta + \lambda_1} & L^2 |_P \rto_-{\delta + \lambda_1} & \ldots
}
$$
also form a quasi-isomorphism of complexes. This completes the proof.
\end{pf}

\section{\'{E}tale fibrations}

\subsection{The statement of main theorem}

Let $\M=(M,L,\lambda)$ and $\N= (N,L',\lambda')$ be $L_\infty$-bundles, and $(f,\phi):\M\to\N$ be an \'{e}tale fibration. Due to Proposition~\ref{prop:IsoToLinearMor}, we may assume that $(f,\phi):\M\to\N$ is linear, i.e. $\phi=\phi_1$. 

\begin{defn}
A morphism $(U,H,\mu) \to (M,L,\lambda)$ of $L_\infty$-bundles is called a \textbf{transfer embedding} if 
\begin{items}
\item
$U$ is an open submanifold of $M$ which contains $\pi_0(\M)$, and 
\item
it is the composition of the trivial fibration $\M|_U \hookrightarrow \M$ with an inclusion morphism obtained by the transfer theorem (Proposition~\ref{transfertheorem}). 
\end{items}

We say $(Y,E,\nu)$ is a {\bf sub-$L_\infty$-bundle} of $(U,H,\mu)$ if $Y$ is a submanifold of $U$, and $E$ is a graded vector subbundle of $H|_Y$ such that the inclusion map $E \hookrightarrow H$ is a linear morphism of $L_\infty$-bundles. 

 Let $Y\subset M$ be a submanifold, and $E^1\subset L^1|_Y$ a subbundle, such that the curvature $\lambda_0|_Y:Y\to L^1|_Y$ factors through $E^1$. Then $\Y=(Y,E^1\oplus L^{\geq2}|_Y)$ is an $L_\infty$-bundle in a canonical way, and the inclusion $\Y\to \M$ is a linear morphism.  We call such embeddings {\bf simple embeddings}. 
\end{defn}

We prove the following theorem in this section.

\begin{thm}\label{recap}
Let $\M=(M,L,\lambda)$ and $\N = (N,L',\lambda')$ be $L_\infty$-bundles, and $(f,\phi): \M \to \N$ a linear fibration. 
\begin{items}
\item[(1)]
If $(f,\phi)$ is \'{e}tale at each point in a subset $Z$ of the Maurer-Cartan locus of $\M$, then there exist an $L_\infty$-bundle $\Y = (Y,E,\nu)$ and a commutative diagram of $L_\infty$-bundles
$$
\begin{tikzcd}
\H \ar[r,dotted] & \M \ar[d] \\
\Y \ar[u,dotted] \ar[ru] \ar[r] & \N
\end{tikzcd}
$$
such that  
\begin{items}
\item
the morphism $\Y \to \M$ is the composition $\Y \to \H \to \M$, where $ \Y \to \H$ is a simple embedding, and $\H = (U,H,\mu) \to \M$ is the composition of a finite sequence of transfer embeddings;

\item
$Z \subset Y \subset M$;
\item
the composition 
$
\H \to \M \to \N
$ 
is a linear fibration  such that $H^{\geq 2} \xrightarrow{\cong} f^*(L')^{\geq 2}$
is an isomorphism of graded vector bundles;
\item
the morphism $\Y \to \N |_{f(Y)}$ is a linear local isomorphism of $L_\infty$-bundles.
\end{items}
\item[(2)]
Furthermore, if $(f,\phi):\M \to \N$ is a trivial linear fibration, then the transfer embeddings and the simple embedding can be chosen so that the morphism $\Y \to \N |_{f(Y)}$ is a linear isomorphism of $L_\infty$-bundles.
\end{items}
\end{thm}
More explicitly, by choosing a splitting of the short exact sequence $0 \to K \hookrightarrow H^1 \xrightarrow{} f^\ast(L')^1 \to 0$ of vector bundles, one has a decomposition
\begin{align*}
H^1 & = K \oplus  f^\ast (L')^1, \\
\mu_0 \,  & = \, u \, + \, f^\ast\lambda_0' \, .
\end{align*}
Then $u$ is a regular section of $K $ over $U$. This means that for every point $P\in M$, such that $u(P)=0$, the derivative induces a surjective
linear map $D_P u :TM|_P\to E|_P$.  In other words, $u$ being a regular section 
is equivalent to that 
$u$ is transversal to the zero section of $E$. One can choose $Y$ to be an open neighborhood of $\pi_0(\M)$ in $Z(u)$ and $E^1 = f^\ast (L')^1|_Y$ such that the restriction $f|_Y$ is a diffeomorphism from $Y$ to an open submanifold of $N$ which contains $\pi_0(\N)$.

We proceed the proof of Theorem~\ref{recap} in three steps: 
\begin{items}
\item[{\bf Step~1.}] 
construct a finite sequence of transfer embeddings $(U,H,\mu) \to \cdots \to  (M,L,\lambda)$ such that the composition 
$$
(U,H,\mu) \to (M,L,\lambda) \to (N,L',\lambda')
$$
is a linear fibration whose restriction $\big(U, H^{\geq 2} \big) \to \big( U, f^\ast (L')^{\geq 2}|_U \big)$ is an isomorphism of graded vector bundles; 
\item[{\bf Step~2.}] 
construct a simple embedding $(Y,E,\nu) \hookrightarrow (U,H,\mu)$ such that $Z \subset Y$ and the composition 
\begin{equation}\label{eq:CompToLocIso}
(Y,E,\nu) \to (U,H,\mu) \to (M,L,\lambda) \to (N,L',\lambda')
\end{equation}
is a linear local isomorphism of $L_\infty$-bundles;

\item[{\bf Step~3.}] 
in the case of trivial fibration, use a topological lemma (Lemma~\ref{lem:TopoLem}) to modify the simple embedding constructed in Step~2 so that the composition \eqref{eq:CompToLocIso} becomes an isomorphism. 
\end{items}

\subsection{Step~1: construction of transfer embeddings}

To achieve Step~1, we construct a sequence of morphisms by the transfer theorem in the following way.

\begin{lem}
\label{firstcase}
Consider a fixed base manifold
$M$, with a linear morphism of $L_\infty$-bundles $(\id_M,\phi): (M,L, \lambda)\to (M,L', \lambda') $ over it.
Let $k\geq2$. Suppose that 
\begin{items}
\item $\phi:L\to  L'$ is an isomorphism of vector bundles in
degrees $\geq k+1$, and an epimorphism in degrees $\leq k$. 
\item Moreover, let $Z\subset \pi_0(\M)$ be a subset of the classical locus
of $\M$ such that, at every point $P$ of $Z$, the induced map on the cohomology 
$H^n(L|_P, \lambda_1)\to H^n(L'|_P, \lambda'_1)$ is an isomorphism when
 $n \geq k$ and
surjective when $n =k-1$.
\end{items}
Then, after restricting to an open neighborhood of $Z$ in $M$, if
necessary,  there exists an inclusion morphism 
$\iota:H\to L$ obtained by the transfer theorem such that the composition $\phi\circ \iota:H\to L'$ is a linear morphism of 
$L_\infty$-bundles, which is an isomorphism in degrees $\geq k$ and an epimorphism in degrees $\leq k-1$.
\end{lem}
\begin{pf}
Since $\phi$ is an epimorphism in each degree, the kernel $K$ of $\phi:L\to L'$ is a graded vector
bundle, and an $L_\infty[1]$-ideal in $L$. 
Consider the diagram
$$\xymatrix{
&K^{k-1}\ar@{^{(}->}[d]^-j\rto^{\lambda_1} & K^k\ar@{^{(}->}[d]^-j\\
L^{k-2}\rto&L^{k-1}\rto^{\lambda_1} & L^k\rto& L^{k+1}\rlap{\,.}}$$
A diagram chase proves that at all points of $Z$, the vector bundle
homomorphism $\lambda_1:K^{k-1}\to K^k$ is surjective.  This
will still be the case in an open neighborhood, so we may assume,
without loss of generality, that this map is an epimorphism over all
of $M$.  We then choose a section $\chi:K^k\to K^{k-1}$ of $\lambda_1$, and a retraction of $\theta:L^k\to K^k$ of $j:K^k\to L^k$, and define
$\eta:L^k\to L^{k-1}$ to be equal to $\eta=j\chi\theta$.  
$$\xymatrix@=3pc{
&K^{k-1}\dto^j\ar[r]^{\lambda_1 } & K^k\dto^j\ar@/^/[l]^\chi\\
L^{k-2}\rto & L^{k-1}\rto_{\delta=\lambda_1} &
L^k\rto\ar@/^/[u]^\theta\ar@/_/[l]_{\eta=j\chi\theta}& L^{k+1}}$$
We also write $\delta$ for $\lambda_1:L^{k-1}\to L^k$, but set
$\delta=0$ elsewhere, to define a differential $\delta:L\to L$ of degree~1.  Our definition ensures that $\delta^2=0$, although
$\lambda_1^2\not=0$. 

One checks that $\eta\delta\eta=\eta$, and so $\eta\delta$ and
$\delta\eta$ are projection operators. Hence, in both cases, kernel and
image are bundles. Let $H^{k-1}=\ker \eta\delta$, and $H^k=\ker
\delta\eta$. For $i\not=k,k-1$, we set $H^i=L^i$. 

The two projection operators 
$j\theta= \delta\eta$ coincide, and hence   $H^k=\ker \delta\eta$
is a complement for $K^k=\im j\theta$.  Therefore,   the
composition $H^k\to L^k\to {L'}^{k}$ is an isomorphism. 

We have $L^{k-1}=\ker \eta\delta+\im \eta \delta\subset
\ker\eta\delta+K^{k-1}=H^{k-1}+K^{k-1}$, because $\eta$ factors
through $K^{k-1}$, by definition.  It follows that the composition
$H^{k-1}\to L^{k-1}\to {L'}^{k-1}$ is surjective. 

Now $\lambda-\delta$ is a curved $L_\infty[1]$-structure on the complex of vector bundles $(L,\delta)$, and $\eta$ is a contraction of $\delta$. We apply the transfer theorem for bundles of curved $L_\infty[1]$-algebras: Proposition~\ref{transfertheorem}.    This gives rise
to a structure $\mu$, of a bundle of curved $L_\infty[1]$-algebras on the complex $(H,\delta)$,
together with a morphism of curved $L_\infty[1]$-structures $\iota:H\to L$, whose linear part is given by
the inclusion $H\subset L$. Here the linear part is the inclusion because $\eta(\lambda_1 -\delta) =0$. 
We consider $\delta+\mu$ as an $L_\infty$-bundle structure on $H$. 

The composition $\phi\bullet\iota$ is linear, because all non-trivial trees involved in $\iota$ have $\eta$ at the root, and $\phi\circ\eta=0$.

In order to apply the transfer theorem, we use the  variation of the
natural filtration at level $k$, see Remark~\ref{varfil},  on
$L$. Both $\delta$ and $\eta$ preserve this filtration, hence $H$
inherits this filtration, and  the transfer theorem applies. 
\end{pf}

\begin{pfStep1}
Since all the bundles are assumed to have finite ranks here, the first  assumption in Lemma~\ref{firstcase} is automatically satisfied for a linear fibration with a large enough $k$. It is also clear that the second assumption is satisfied for \'{e}tale fibrations. Applying Lemma~\ref{firstcase} inductively to the given linear fibration, we obtain a sequence of transfer embeddings $(U,H,\mu) \to \cdots \to (M,L,\lambda)$ satisfying the condition in Step~1.
\end{pfStep1}

\subsection{Step~2: splitting}

\begin{lem}\label{lastcase}
Let $(f,\phi):(M,H,\mu)\to (N,L',\lambda')$ be a linear fibration of $L_\infty$-bundles such that  $\phi:H^k\to f^\ast (L')^k$ is an isomorphism, for all $k\geq2$.
Let $Z$ be a subset of the Maurer-Cartan locus of $(H,\mu)$,
such that the morphism of $L_\infty$-bundles $(f,\phi)$ is \'etale at
every point of $Z$. 
After restricting to an open neighborhood $U$ of $Z$ in $M$, there exists a
submanifold $Y$ of $U$, and subbundle $E^1$ of $H^1|_Y$, such that 
\begin{items}
\item $Z\subset Y$,
\item the restriction of the curvature $\lambda_0|_Y$ is contained in $\Gamma(Y,E^1)\subset \Gamma(Y,H^1|_Y)$, so that $E:= E^1\oplus H^{\geq2}|_Y$ is a bundle of curved $L_\infty[1]$-subalgebras of $H|_Y$, 
\item   the composition $Y\to M\to N$ is  \eetale,
\item    the map $\phi|_Y:E^1|_Y\to f^\ast (L')^1|_Y$ is an isomorphism of vector bundles, so that the composition $(Y,E)\to (Y,f^\ast L'|_Y)$ is a linear isomorphism of bundles of curved $L_\infty[1]$-algebras.
\end{items}
In particular, the composition $(Y,E)\to (N, L')$ is \'etale at all points of $Z\subset Y$.  
\end{lem}
\begin{pf}
We have a diagram
$$\xymatrix{
M\rto^-s\dto_f & H^1\dto^\phi   \\
N\rto^-t & (L')^1}$$
Here $f:M\to N$ is a submersion, and $\phi:H^1\to f^\ast (L')^1$ an
epimorphism. We have written $s=\mu_0$ and $t=\lambda_0'$ for the respective curvatures. 

We denote by $K$  the kernel of $\phi:H^1\to f^\ast (L')^1$.  We choose 
\begin{items}
\item a linear connection on the vector bundle $(L')^1$; 
\item a retraction  $\theta:H^1\to K = \ker(\phi)$ of the inclusion 
$K\to   H^1$, giving rise to a splitting $H^1=K\oplus \tilde E^1$, where $\tilde E^1 \subset H^1$ is a subbundle such that $\phi:\tilde E^1 \to f^\ast (L')^1$ is an isomorphism;
\item a linear connection on the vector bundle $K$.
\end{items}
These data give rise to a linear connection on $H^1 \cong K \oplus f^\ast(L')^1$.

The curvature $s \in \Gamma(M,H^1)$ splits into a sum 
\begin{equation}\label{eq:DecompCurv}
s=u+f^\ast t,
\end{equation}
where $u\in \Gamma(M,K)$ is a section of $K$.

Consider the diagram of vector bundles over $M$:
\begin{equation}\label{diagram1_lastcase}
\begin{split}
\xymatrix{
T_{M/N}\dto_j \rto^{(\nabla u)\circ j}& K\dto\\
TM\rto^{\nabla s}\dto\urto^{\nabla u} & H^1\dto_\phi\ar@/^/[u]^\theta\\
f^\ast TN\rto^{f^\ast(\nabla t)} & f^\ast
(L')^1\rlap{\,.}}
\end{split}
\end{equation}
Both squares with downward pointing arrows commute. Moreover,  we have  
 $\theta\circ(\nabla s)=\nabla u$. 

Let $P\in Z$. The fact $(f,\phi)$ is \'etale at $P$ means that   in the morphism of short
exact sequences of vector spaces
\begin{equation}\label{diagram2_lastcase}
\begin{split}
\xymatrix{
T_{M/N}|_P\rto\dto & K|_P\dto\\
TM|_P\rto \dto& H^1|_P\dto\\
TN|_{\phi(P)}\rto & (L')^1|_{\phi(P)}\rlap{\,,}}
\end{split}
\end{equation}
the map $T_{M/N}|_P\to K|_P$ is an isomorphism.

Note that the composition $(\nabla u)\circ j$ is an isomorphism at $P\in Z$, because at every point in $Z$, the diagram \eqref{diagram1_lastcase} coincides with the diagram \eqref{diagram2_lastcase}.  Then $(\nabla u)\circ j$  is still an isomorphism in a neighborhood of $Z$, so we
will assume that it is true everywhere: $(\nabla u)\circ j$ is an
isomorphism.

We note that $u \in \Gamma(M, K)$ is a regular section, because at points where $u$ vanishes, the derivative $TM|_P \to K|_P$ coincides with $(\nabla u)|_P$, and $\nabla u : TM \to K$ is an epimorphism since $(\nabla u) \circ j$ is an isomorphism.

Hence, $Y=Z(u)$ is a submanifold of $M$, and we have a short exact sequence
of vector bundles $TY\to TM|_Y\to K|_Y$. Note that $Z\subset Y$.

Since $T_{M/N}|_Y$ is a complement for $TY$ in $TM|_Y$, it follows
that the composition $TY\to TM|_Y\to (f^\ast TN)|_Y$ is an isomorphism, so
that the composition $Y\to M\to N$ is \'etale. 

By definition $u|_Y=0$, so after restricting to $Y$, the curvature $s|_Y$ is contained in the subbundle $E^1 := \tilde E^1 |_Y\subset H^1|_Y$. 
\end{pf}

\begin{pfStep2} 
By Lemma~\ref{lastcase}~(ii), there exists a simple embedding $(Y,E,\nu) \hookrightarrow (U,H,\mu)$. The other properties in Lemma~\ref{lastcase} guarantee that all the requirements in Step~2 are satisfied.
\end{pfStep2}

\subsection{Step~3 and the proof of Theorem~\ref{recap}}

To prove Theorem~\ref{recap}~(2), assume that the given linear fibration $(f,\phi):(M,H,\mu)\to (N,L',\lambda')$ is trivial. 
In Lemma~\ref{lastcase}, we proved that the composition $Y \to M \to N$ is a local diffeomorphism. However, to get an isomorphism of $L_\infty$-bundles, we need to show that $f:Y \to f(Y) \subset N$ is a diffeomorphism. To achieve this step, we need a topological lemma.

We believe the following lemma is classical, but we could not find a reference. So we include a proof here for the sake of completeness. 
 
\begin{lem}\label{lem:TopoLem}
Let $f:M\to N$ be  a local diffeomorphism, and $Z\subset N$ a closed subset, with preimage $X=f^{-1}(Z)\subset M$. Assume that the induced map $X\to Z$ is injective.  Then there exists an open neighborhood $V$ of $X$ in $M$, such that $f|_V:V\to N$ is a diffeomorphism onto an open neighborhood of $Z$ in $N$. 
\end{lem}
\begin{pf}
It suffices to prove that there exists an open neighborhood $V$ of $X$ in $M$, such that $f|_V:V\to N$ is injective. 

{\bf Case I\@. } $Z$ is compact.

Choose a sequence of relatively compact open neighborhoods $V_1\supset V_2\supset\ldots$ of $X$ in $M$, such that 
$$\bigcap_i V_i=X\,.$$
We claim that there exists an $i$, such that $f$ is injective when restricted to $V_i$. If not, choose in every $V_i$ a pair of points $(P_i,Q_i)$, such that $f(P_i)=f(Q_i)$. Upon replacing $(P_i,Q_i)$ by a subsequence, we may assume that $\lim_{i\to\infty} P_i=P$, and $\lim_{i\to\infty}Q_i=Q$. We have $P$ and $Q\in X$, and since $f$ is continuous, $f(P)=f(Q)$.  This implies that $P=Q$. Let $V$ be a neighborhood of $P=Q$ in $M$, such that $f|_V$ is injective.  For sufficiently large $i$, both $P_i$ and $Q_i$ are in $V$, which is a contradiction.

{\bf Case II\@. } General case. 

Let $(U_i)_{i\in I}$ be a locally finite open cover of $N$, such that all $U_i$ are relatively compact. It follows that, for every $i\in I$, the set
$$I_i=\{j\in I\st  U_j\cap U_i\not=\varnothing\}$$
is finite.  Hence,
$$\tilde U_i=\bigcup_{j\in I_i}U_j$$
is still relatively compact, and by Case~I, there exists an open neighborhood $\tilde V_i$ of $X\cap f^{-1}(\tilde U_i)$ on which $f$ is injective. We may assume that $\tilde V_i\subset f^{-1} (\tilde U_i)$. Define 
$$V_i= f^{-1}(U_i) \cap\bigcap_{j\in I_i}\tilde V_j\,.$$
This is an open subset of $M$.  If $P\in X$, such that $f(P)\in U_i$, then $f(P)\in \tilde U_j$, for all $j\in I_i$.  Hence, $P\in f^{-1}(\tilde U_j)\subset \tilde V_j$, for all $j\in I_i$, and so $P\in  V_i$. Thus, $X\cap f^{-1}(U_i)\subset V_i$.

We define
$$V=\bigcup_{i\in I} V_i\,.$$
Then $V$ is an open neighborhood of $X$ in $M$. We claim that $f$ is injective on $V$. Let $P,Q\in V$ be any two points such that $f(P)=f(Q)$. Suppose $P\in V_i$ and $Q\in V_j$. Then $f(P)\in U_i$, and $f(Q)\in U_j$, so that $U_i$ and $U_j$ intersect. Hence $i\in I_j$, and therefore $V_j\subset \tilde V_i$. Since we have $V_i\subset\tilde V_i$, we have that both $P,Q\in \tilde V_i$. Since $f$ is injective on $\tilde V_i$, we conclude that $P=Q$. 
\end{pf}

Now, we are ready to complete the proof of Theorem~\ref{recap}.

\begin{pfrecap}
We have constructed a sequence of transfer embeddings $(U,H,\mu) \to \cdots \to (M,L,\lambda)$ and a simple embedding $(Y,E,\nu) \hookrightarrow (U,H,\mu)$ according to Lemma~\ref{firstcase} and Lemma~\ref{lastcase}. This confirms the assertion~(1) in Theorem~\ref{recap}. 

To prove Theorem~\ref{recap}~(2), we need the following observation from the proof of Lemma~\ref{lastcase}: 
Let $X$ be the preimage of $f(Z)$ under $f$ in $Y=Z(u)$. 
If $(f,\phi)$ is a trivial fibration,  then the restriction $f:X \to f(Z)$ is a bijection because of Decomposition \eqref{eq:DecompCurv}. 

Thus, the underlying smooth map $f:Y  \to N$ satisfies the assumption in Lemma~\ref{lem:TopoLem}. Consequently, there exists an open neighborhood $V$ of $Z$ in $Y$ such that $(f,\phi):(Y,E,\nu)|_V \to (N,L',\lambda')|_{f(V)}$ is a linear isomorphism of $L_\infty$-bundles. This completes the proof.
\end{pfrecap}

\section{Inverse function theorem and the homotopy category of $L_\infty$-bundles} 

One important application of Theorem~\ref{recap} is the inverse function theorem for $L_\infty$-bundles (Theorem~\ref{ift}), which allows us to give a simple description of the homotopy category of $L_\infty$-bundles. 

\subsection{The inverse function theorem for $L_\infty$-bundles}

\begin{thm}\label{ift}
Let $\M=(M,L,\lambda)$ and $\N=(N,L',\lambda')$ be $L_\infty$-bundles, and
$(f,\phi):\M\to\N$ a fibration.  Let $Z\subset \pi_0(\M)$ be a subset
of points at which $(f,\phi)$ is \'etale. Then
there exists  \eetale \
  of manifolds  $Y\to N$ and a commutative diagram
$$\xymatrix{
Z\rto\dto & \M\dto\\
\Y\rto\urto & \N\rlap{\,.}}$$
Here $\Y$ is the pullback of  $(L',\lambda')$ via $Y\to N$, and
the lower triangle consists of   morphisms of $L_\infty$-bundles.   

If, moreover,  $\pi_0(\M) \to\pi_0(\N)$ is injective, we can choose $Y$ to be an open
 submanifold  of $N$.
\end{thm}
\begin{pf}
Without loss of generality, we may assume that the fibration $(f,\phi):\M\to\N$ is linear, i.e. $\phi=\phi_1$. 
By applying Theorem~\ref{recap} to $(f,\phi)$, we obtain an $L_\infty$-bundle $\Y =(Y,E,\nu) \cong (Y, f^\ast L', f^\ast \lambda')$ together with a local diffeomorphism $Y \to N$ and a morphism of $L_\infty$-bundles $\Y \to \M$. This $L_\infty$-bundle $\Y$ proves the first part of Theorem~\ref{ift}.

To finish the proof, we now assume that $\pi_0(\M) \to \pi_0(\N)$ is injective. Since $Y$ was chosen to be the zero locus of $u$ in Lemma~\ref{lastcase}, it follows from Equation~\eqref{eq:DecompCurv} that $f^{-1}\big(Z(t)\big) \cap Y \subset Z(s)$, and thus $f$ maps $f^{-1}\big(Z(t)\big) \cap Y$ injectively to $Z(t)$. Therefore, by Lemma~\ref{lem:TopoLem}, the space $Y$ can be chosen to be an open submanifold of $N$. This concludes the proof.
\end{pf}

\begin{cor}\label{cor:LocalSec}
\begin{items}
\item
A fibration admits local sections through every point at which it is \'etale.
\item
Every trivial fibration of $L_\infty$-bundles admits a section, after
restricting the target to an open neighborhood of its classical
locus. 
\end{items}
\end{cor}
\begin{pf}
Apply the theorem either with $Z$ a single point or with $Z = \pi_{0}(\M)$.
\end{pf}

\begin{rmk}
Let $\M$ and $\N$ be smooth manifolds (i.e. $L = M \times \{0\}$ and $L' = N \times \{0\}$). Then a morphism $(f,\id):\M \to \N$ is \'etale at $P$ if and only if the tangent map of $f$ is an isomorphism at $P$. In this extreme case, Corollary~\ref{cor:LocalSec} reduces to the inverse function theorem for smooth manifolds.
\end{rmk}

\subsection{Remarks on the homotopy category}

In this section, we study  a special type of categories of fibrant objects (see Property~\eqref{eq:CFOProperty}) and their homotopy categories.

Let $\cC$ be an arbitrary category of fibrant objects. The \textbf{homotopy category} $\Ho(\cC)$ of $\cC$ is the localization of $\cC$ at the  weak equivalences.  It follows from Brown's lemma (every weak equivalence can be factored into a composition of a trivial fibration and a section of a trivial fibration) that  $\Ho(\cC)$ is same as the localization at the  trivial fibrations.

In general, it is complicated to describe morphisms in a localization by an arbitrary class $W$ of morphisms. Nevertheless, if $W$ admits a calculus of fractions (see \cite{MR0210125}), then there is a simple description of the localization in terms of spans.  In \cite{MR341469}, Brown first approximated $\Ho(\cC)$ by another category $\pi\cC$ and then showed that the weak equivalences in $\pi\cC$ have a calculus of right fractions. In this way, he obtained a more explicit description of $\Ho(\cC)$. He also remarked in \cite{MR341469}, that one can replace weak equivalences by trivial fibrations in his description of $\Ho(\cC)$.

 More explicitly, recall that a {\bf path space object} for an object
 $Y$ in $\cC$ is an object $PY$ together with morphisms $Y \xrightarrow{s} PY \xrightarrow{e} Y \times Y$, where $s$ is a weak equivalence, $e$ is a fibration, and the composition $e \circ s$ is the diagonal morphism. The fibration $e:PY \to Y \times Y$ decomposes into two evaluation maps: $\ev_0 = \pr_1 \circ e$ and $\ev_1 = \pr_2 \circ e$. 
 
\begin{defn}[\cite{MR341469}] 
Two morphisms $f,g: X \to Y$ are \textbf{homotopic}, denoted by $f\simeq g$,
 if for some path space object $PY$ for $Y$, there exists a morphism $h:X \to PY$, such that $f= \ev_0 \circ h$ and $g = \ev_1 \circ h$. 
In this case,  $h$ is called a \textbf{homotopy} from $f$ to $g$.
\end{defn}

\begin{lem}\label{sect:triv}
Let $s$ be a section of a trivial fibration $t:X'\to X$.  Then $st\simeq 1$. 
\end{lem}
\begin{pf}
The diagonal morphism $X'\to X'\times_X X'$ is a weak equivalence. Hence, when we factor the canonical map $X'\times_X X'\to X'\times X'$ as a weak equivalence $w$, followed by a fibration $f$, we obtain a path space for $X'$:
$$\begin{tikzcd}
&& PX'\arrow[dr,"f"]\\
X'\arrow[r,"\Delta"] & X'\times_X X'\arrow[rr]\arrow[ur,"w"]&&X'\times X'\rlap{\,.}
\end{tikzcd}$$ 
The morphism $st\times 1:X'\to X'\times X'$ factors through $X'\times_X X'$, because $tst=t$. So we can replace $\Delta$ by $st\times 1$ in the above diagram:
$$\begin{tikzcd}
&& PX'\arrow[dr,"f"]\\
X'\arrow[r,"st\times1"] & X'\times_X X'\arrow[rr]\arrow[ur,"w"]&&X'\times X'\rlap{\,,}
\end{tikzcd}$$
and we see that, indeed, $st\simeq1$. 
\end{pf}

Let $\pi\cC$ be the category with the same objects  as $\cC$ and with morphisms 
$$\Mor_{\pi \cC}(X,Y) = \Mor(X,Y)\big/\approx\, ,$$ 
where $f \approx g$ if and only if there exists a {weak equivalence} $t:X' \to X$ such that $f \circ t \simeq g \circ t$. The relation $\approx$ is an equivalence relation, and if $f\approx g$, then $f$ is a weak equivalence if and only if $g$ is.  Thus the category $\pi\cC$ inherits the notion of weak equivalence.  Brown \cite{MR341469} proved that $\pi\cC$ admits a right calculus of fractions with respect to weak equivalence, and that the category of right fractions is the homotopy category of $\cC$.

We will now assume that $\cC$ satisfies the following property:
\begin{equation}\tag{\textasteriskcentered }\label{eq:CFOProperty}
\text{Every trivial fibration admits a section.}
\end{equation}

A first consequence of this condition is the following.

\begin{lem}\label{first:cond}
In a category of fibrant objects satisfying {\rm ($\ast$)},  $f\simeq g$ implies $ uf\simeq ug$ and 
$fv\simeq gv$, whenever these compositions make sense. 
\end{lem}
\begin{pf}
The claim $fv\simeq gv$ is trivial.  By \cite[Proposition~1~(ii)]{MR341469},  the assumption $f\simeq g$ implies that for every morphism $u$, such that $uf$ and $ug$ are defined, there exists a trivial fibration $x$, such that $ufx\simeq ugx$.  Precomposing with a section of $x$ gives the required result $uf\simeq ug$. 
\end{pf}

\begin{lem}\label{fracc}
Under the assumption $\text{\rm($\ast$)}$, every right fraction is equivalent to a fraction with an identity in the denominator. Thus, every morphism in $\Ho(\cC)$ can be represented by a morphism in $\cC$, with $f$ and $g$ defining the same morphism in $\Ho(\cC)$, if and only if $f\approx g$. 
\end{lem}
\begin{pf}
Let $w$ be a weak equivalence, and $f$ a morphism.  Consider the right fraction
$$\begin{tikzcd}
&X'\arrow[dl,"w"']\arrow[drr,"f"]\\
X &&& Y
\end{tikzcd}$$
Use Brown's lemma \cite[Factorization lemma]{MR341469} to factor $w=us$, where $u$ is a trivial fibration, and $s$ a section of a trivial fibration $t$, so that $ts=1$.  Then use Property~($\ast$), to find a section $x$ of $u$, so that $ux=1$.  The  diagram 
$$\begin{tikzcd}
&X'\arrow[dl,"w=us"']\arrow[drr,"f"]\arrow[d,"s"]\\
X &\arrow[l,"u"'] X''\arrow[rr,"ft"]&& Y\\
&X\arrow[lu,"1"]\arrow[u,"x"]\arrow[urr,"ftx"']
\end{tikzcd}$$
shows that our given right fraction is equivalent to the morphism $ftx:X\to Y$. 
\end{pf}

\begin{lem}\label{eq.rel}
Under the assumption $\text{\rm($\ast$)}$, for any two morphisms $f,g:X\to Y$ in $\cC$, we have $f\approx g$ if and only if $f\simeq g$. 
\end{lem}
\begin{pf}
The main part consists of fleshing out \cite[Remark~2 on page~425]{MR341469}.  Let $w:X'\to X$ be a weak equivalence, such that $fw\simeq gw$. Again, we use Brown's lemma to factor $w=us$, where $u$ is a trivial fibration and $s$ is a section of a trivial fibration $t$. 
Thus, we have $fus\simeq gus$, and hence $fust\simeq gust$ by Lemma~\ref{first:cond}. 

Now by Lemma~\ref{sect:triv}, we have $1\simeq st$, which implies $fu\simeq fust\simeq gust\simeq gu$, by Lemma~\ref{first:cond}, and transitivity of the homotopy relation.  Precomposing with a section of $u$,   we obtain $f\simeq g$, as required. 
\end{pf}

Now Lemmas~\ref{fracc} and~\ref{eq.rel} imply immediately the following 

\begin{prop}
Under the assumption {\rm($\ast$)}, the homotopy category $\Ho(\cC)$ of $\cC$ has the same objects as $\cC$, and has homotopy classes of morphisms in $\cC$ as morphisms:
$$\Hom_{\Ho(\cC)}(X,Y)=\Hom_\cC(X,Y)/\simeq\,.$$
\end{prop}

\begin{cor}
Under the assumption {\rm($\ast$)}, every weak equivalence is a homotopy equivalence. 
\end{cor}
\begin{pf}
Factor the weak equivalence $w=us$, where $u$ is a trivial fibration and $s$ is a section of the trivial fibration $t$. Let $v$ be a section of $u$. Then we have $ts=1$, $st\simeq 1$, and $uv=1$, $vu\simeq1$. Let $y=tv$.  Then $y$ is a homotopy inverse of $w$.  In fact, $wy=ustv\simeq uv=1$, and $yw=tvus\simeq ts=1$. 
\end{pf}

\begin{rmk}
Under the assumption {\rm($\ast$)}, two morphisms $f,g:X\to Y$ are homotopic, if and only if for {\em every }path space object $PY$ for $Y$ there exists a homotopy $h:X\to PY$, such that $\ev_0\circ h=f$, and $\ev_1\circ h=g$. 
\end{rmk}

\subsection{The homotopy category of $L_\infty$-bundles}

Now let $\catL$ be the category of $L_\infty$-bundles. 
To investigate $\Ho(\catL)$, we introduce the {\bf category of germs of $L_\infty$-bundles} $\germL$, whose objects are  $L_\infty$-bundles, and whose morphisms are given by
$$
\Mor_{\germL}(\M,\N) = \injectlim_{\pi_0(\M)\subset U}\Mor_{\catL}(\M|_U,\N)\,,
$$
where $\M$ and $\N$ are $L_\infty$-bundles, and the colimit is taken over all open neighborhoods of the classical locus of $\M$ inside the underlying manifold of $\M$. In other words, a morphism $[f_U]:\M \to \N$ in $\germL$ is an equivalence class represented by a morphism of $L_\infty$-bundles $f_U:\M|_U \to \N$, where $U$ is an open neighborhood of $\pi_0(\M)$ in the base manifold of $\M$. Two equivalence classes $[f_U], [f_V]: \M \to \N$ are equal if and only if there exists another open neighborhood $W$ of $\pi_0(\M)$ such that $W \subset U \cap V$ and $f_U|_W = f_V|_W$ in $\catL$. We say that $[f_U]$ is the {\bf germ} of $f_U$.

Composition of germs, and identity germs, are defined in a straightforward manner, leading to the category $\germL$, which comes together with a canonical functor $\catL\to\germL$, sending  a morphism to its germ. (In fact,  $\germL$ is the localization of $\L$ at the weak equivalences of the form $\M|_U\to\M$.)

A morphism in $\germL$ is called a   \textbf{weak equivalence}, if it is a germ of a weak equivalence. This definition makes $\germL$ into a category with weak equivalences.  Since every morphism $\M|_U\to \M$ is a weak equivalence in $\catL$, the homotopy category of $\germL$ is equal (or rather canonically isomorphic, not just equivalent) to the homotopy category of $\catL$:
$$\Ho(\catL)=\Ho(\germL)\,.$$

We call a morphism in $\germL$  a \textbf{fibration},  if it is a germ of a fibration. (Not every representative of a fibration in $\germL$ needs to be a fibration in $\catL$, in contrast to the situation with weak equivalences.)

\begin{prop} 
The category $\germL$ is a category of fibrant objects, which satisfies Property~{\rm ($\ast$)}. 
\end{prop}
\begin{pf}
To prove that fibrations and trivial fibrations are stable under pullbacks, it is helpful to note that fibered products in $\catL$ map to fibered products in $\germL$.  In fact, none of the axioms of a category of fibrant objects is difficult to check.  Property~{\rm ($\ast$)} follows directly from Corollary~\ref{cor:LocalSec}.
\end{pf}

The functor $\catL\to \germL$ maps path space objects to path space objects. Conversely, we have 
\begin{lem}\label{estra:lem}
Every path space object for $\N$ in $\germL$ is isomorphic to the image under $\catL\to\germL$ of a path space object in $\catL$ for $\N|_U$, for some open neighborhood $U$ of $\pi_0(\N)$. 
\end{lem}
\begin{pf}
Suppose that we have a path space object for $\N$ in the category $\germL$, 
\begin{equation}\label{psobjjkl}
\begin{tikzcd}
& \P_{\text{germ}}\N\arrow[dr,two heads,dotted]\\
\N\arrow[rr,"\Delta",dotted]\arrow[ur,"\sim",dotted] &  & \N\times \N\rlap{\,,}
\end{tikzcd}\end{equation}
where the dotted arrows indicate germs of morphisms.  By replacing $\P_{\text{germ}}\N$ by a neighborhood of its classical locus, we may assume that $\P_{\text{germ}}\N\to\N\times\N$ is a morphism in $\catL$. Then we can find a neighborhood $U$ of $\pi_0(\N)$, that is a common domain of definition for both dotted arrows originating at $\N$.  We obtain a commutative diagram in $\catL$
$$\begin{tikzcd}
& \P_{\text{germ}}\N\arrow[dr,two heads]\\
\N|_U\arrow[r,"i"]\arrow[ur,"\sim"] & \N \arrow[r,"\Delta"]& \N\times \N\rlap{\,,}
\end{tikzcd}$$
where $i$ is the inclusion of   $\N|_U$  into $\N$, the morphism $\N|_U\to\P_{\text{germ}}\N$ is a weak equivalence, and $\P_\text{germ}\N\to \N\times \N$ is a fibration. Now $U\times U$ is an open neighborhood of the classical locus of $\N\times \N$, and restricting to the preimage of $U\times U$ in $\P_{\text{germ}}\N$, we may assume that we have the diagram 
$$\begin{tikzcd}
& \P_{\text{germ}}\N\arrow[dr,two heads]\\
\N|_U\arrow[rr,"\Delta"]\arrow[ur,"\sim"] & & \N|_U\times \N|_U\rlap{\,}
\end{tikzcd}$$
in $\catL$. It represents a path space object for $\N|_U$ in $\catL$, whose image in $\germL$ is isomorphic to~(\ref{psobjjkl}).
\end{pf}

\begin{cor}\label{cor:htpCat}
For $L_\infty$-bundles $\M$ and $\N$, we have
$$
\Mor_{\Ho(\catL)}(\M,\N) = \Mor_{\germL}(\M,\N)\big/ \simeq .
$$
More explicitly, every morphism $\M\to \N$ in $\Ho(\L$) can be represented by a morphism of $L_\infty$-bundles $\M|_U\to \N$, with $U$ being an open neighborhood of $\pi_0(\M)$ in the underlying base manifold, and two such are equal in $\Ho(\L)$, if and only if, when restricted to  suitable neighborhoods of $\pi_0(\M)$ and $\pi_0(\N)$, they become homotopic. 
\end{cor} 
\begin{pf}
For the last claim, use Lemma~\ref{estra:lem}.
\end{pf}

\begin{cor}
In the category $\germL$, every weak equivalence is a homotopy equivalence.  
\end{cor}

\begin{rmk}
The question remains, if for any weak equivalence $f:\M\to\N$ in $\catL$, we can find open neighborhoods $\pi_0(\M)\subset U$ and $\pi_0(\N)\subset V$, such that $f$ induces a morphism $f:\M|_U\to\N|_V$, which admits a homotopy inverse in $\catL$. (A similar statement has been conjectured by Amorim-Tu~\cite{MR4542391}, although it is not clear to us whether their notion of homotopy is equivalent to ours.)
\end{rmk}

\section{Quasi-isomorphisms}

As another application of Theorem~\ref{recap}, in this section, we investigate the relationship between weak equivalences and quasi-isomorphisms.

Recall that one has a fully faithful functor \cite{BLX1}
\begin{align*}
(\text{$L_\infty$-bundles})&\longrightarrow (\text{dg manifolds of
positive amplitudes})\\
(M,L,\lambda)&\longmapsto (M,\Sym L^\vee,Q_\lambda)\,.\nonumber
\end{align*}
In particular, a morphism of $L_\infty$-bundles $(M,L,\lambda) \to (N,L',\lambda')$ induces a morphism of dg algebras $\Gamma(N,\Sym (L')^\vee)\to\Gamma(M,\Sym L^\vee)$ of the global sections in the opposite direction. 

\begin{defn}
A morphism of $L_\infty$-bundles is said to be a \textbf{quasi-isomorphism} if it induces a quasi-isomorphism of the dg algebras of global sections.
\end{defn}

In this section, we give an elementary proof of the following

\begin{thm}\label{thm:WeakEqImplyQuasiIso}
Suppose that $(M,L,\lambda)\to (N,L',\lambda')$ is a weak equivalence of $L_\infty$-bundles.  Then the induced morphism of differential graded algebras $\Gamma(N,\Sym (L')^\vee)\to\Gamma(M,\Sym L^\vee)$ is a quasi-isomorphism.
\end{thm}

\subsection{Proof of Theorem~\ref{thm:WeakEqImplyQuasiIso}: transversal case}

Our proof of Theorem~\ref{thm:WeakEqImplyQuasiIso} is divided into two parts. 
We first prove it in the transversal case, and then reduce the general situation to the transversal case. 

Let $(M,E,u)$ be a quasi-smooth $L_\infty$-bundle, where $M$ is a manifold and $u\in \Gamma(E)$ a regular section. In other words, the section $u$ is transversal to the zero section. 
Let $Y\subset M$ be the zero locus of $u$, which is a submanifold of $M$. We have a canonical epimorphism of vector bundles $Du|_Y:TM|_Y\to E|_Y$, whose kernel is $TY$.  

We will place $E$ in degree $1$, so that $(M,E,u)$, as well as $(Y,0,0)$, are $L_\infty$-bundles, and $(Y,0,0)\to (M,E,u)$ is a morphism of $L_\infty$-bundles (in fact, a weak equivalence of $L_\infty$-bundles).  The following proposition contains the well-known Koszul resolution.  We sketch a proof here for completeness.  

\begin{prop}\label{koszul}
The restriction map $(\Gamma(M,\Sym E^\vee),\iota_u )\to (C^\infty(Y),0)$ is a quasi-isomorphism.
\end{prop}
\begin{pf}
Let $ U\supset Y$ be an open neighborhood of $Y$ in $M$. 

\paragraph{Claim.} The restriction map
$$(\Gamma(M,\Sym E^\vee),\iota_u)\longrightarrow (\Gamma(U,\Sym E^\vee),\iota_u)$$ is a quasi-isomorphism.
To prove this, let $f:M\to\rr$ be a differentiable function such that $f|_Y=1$, and $f|_{M\setminus U}=0$. 
Multiplication by $f$ defines a homomorphism of differential graded  
$(\Gamma(M,\Sym E^\vee),\iota_u)$-modules 
$$(\Gamma(U,\Sym E^\vee),\iota_u)\longrightarrow (\Gamma(M,\Sym E^\vee),\iota_u)\,,$$
in the other direction, which is a section of the restriction map.  
Now it suffices to show that for every $i\leq0$, multiplication by $f$ induces the identity on $H^i(\Gamma(M,\Sym E^\vee),\iota_u)$.  This follows from the fact that $f$ restricts to the identity in $C^\infty(Y)$, and that every cohomology group $H^i(\Gamma(M,\Sym E^\vee), \iota_u)$ is a $C^\infty(Y)$-module. 

By the claim, we can replace $M$ by any open neighborhood of $Y$. 
In fact, we will assume that $M$ is a tubular neighborhood of $Y$, with Euler vector field $v$ and projection $\rho:M\to Y$. We denote the relative tangent bundle by $T_{M/Y}$. The Euler vector field is a regular section of $T_{M/Y}$, with zero locus $Y$. 

\paragraph{Claim.} After restricting to a smaller neighborhood of $Y$ if necessary, there exists an isomorphism of vector bundles $\Psi:T_{M/Y}\to E$, such that $\Psi(v)=u$. 

To prove this claim, first assume that there exists a normal coordinate system $(x^i)_{i=0,\ldots,n}$ on $M$, compatible with the tubular neighborhood structure.  This means that the Euler vector field has the form 
$v=\sum_{i=1}^k x^i\del_i$, where $k$ is the rank of $E$, i.e.\ the codimension of $Y$ in $M$. Also assume that $E$ is trivial, with frame $(e_i)_{i=1,\ldots,k}$. Assume that $u=\sum u^ie_i$. Define a bundle map $\Psi:T_{M/Y}\to E$ over $Y$ which, in coordinates, are given by the matrix 
$$\Psi_i^j=\int_0^1(\del_iu^j)(tx^1,\ldots,tx^k,x^{k+1},\ldots,x^n)\,dt\,.$$
Then $\Psi(v)=u$, and $\Psi|_Y$ is the canonical isomorphism $Du|_Y$. 

For the general case, construct a family of $\Psi_\alpha$ locally, and define a global bundle map $\Psi:T_{M/Y}\to E$ using a partition of unity $(\psi_\alpha)$:
$$\Psi=\sum_{\alpha}\psi_\alpha \Psi_\alpha\,.$$
We have that $\Psi(v)=u$, and $\Psi|_Y=D u|_Y$.  Hence $\Psi$ is an isomorphism in an open neighborhood of $Y$.  Upon replacing $M$ by this neighborhood, we may assume that $\Psi$ is an isomorphism globally.

We are now reduced to the case where $E=T_{M/Y}[-1]$, and $u$ is the Euler vector field. In this case, $\Sym \dual E = \Sym(T_{M/Y}^\vee[1]) \cong \Lambda T_{M/Y}^\vee$.  
We define a contraction operator  $\eta:\Gamma(M,\Lambda T_{M/Y}^\vee) \to \Gamma(M,\Lambda T_{M/Y}^\vee)$ by the formula 
$$\eta(\omega)= \int_0^1 \sigma^\ast d\omega\wedge \frac{dt}{t}\,,$$
where $\sigma:[0,1]\times M\to M$ is the multiplicative flow of $v$ in the tubular neighborhood $M$.  

\paragraph{Claim.} $[\eta,\iota_v]=\id - \rho^\ast\circ \iota^\ast$.  

Here, $\rho: \M \to Y$ is the projection, $\iota:Y\to \M$ is the inclusion morphism;  $\rho^\ast: C^\infty(Y) \to \Gamma(M,\Lambda T_{M/Y}^\vee)$ and  $\iota^\ast:\Gamma(M,\Lambda T_{M/Y}^\vee) \to C^\infty(Y)$ are the corresponding induced maps on function algebras.

The claim can be checked locally. So we may assume that we have a normal coordinate system $(x^i)$, as above.  In this case, the multiplicative flow is given by 
$$\sigma(t,x^1,\ldots, x^n)=(tx^1,\ldots tx^k,x^{k+1},\ldots,x^n)\,.$$
The claim follows from a straightforward verification.

Thus, $\big(C^\infty(Y), 0 \big)$ is indeed homotopy equivalent to $\big(\Gamma(M,\Lambda T_{M/Y}^\vee),\iota_v\big)$. This completes the proof.
\end{pf}

\subsection{Proof of Theorem~\ref{thm:WeakEqImplyQuasiIso}: general case}

Let $(M,L,\lambda)$ be an arbitrary $L_\infty$-bundle.  Suppose that $L^1$ is split into a direct sum $L^1=E\oplus H$, in such a way that the curvature $\lambda_0\in \Gamma(M,L^1)$ splits as $\lambda_0=u+s$, where $u$ is a regular section of $E$. Let $Y=Z(u)$ be the vanishing locus of $u$, and consider the induced morphism of $L_\infty$-bundles
$$(Y,H|_Y\oplus L^{\geq2}|_Y,\lambda)\longrightarrow (M,L,\lambda)\,.$$

\begin{lem}\label{former}
The restriction map $$\Gamma(M,\Sym L^\vee,Q_\lambda)\longrightarrow\Gamma\big(Y,\Sym(H^\vee|_Y\oplus (L^{\geq2})^{\vee}|_Y),Q_\lambda\big)$$ is a quasi-isomorphism.
\end{lem}
\begin{pf}
Let us write $\tilde L=H\oplus L^{\geq2}$. 
By the decomposition $L = E \oplus \tilde L$ of graded vector bundles, we obtain an induced morphism of sheaves of differential graded algebras $\Sym E^\vee\to
\Sym L^\vee$ over $M$. Since $\Sym L^\vee$ is locally free over $\Sym E^\vee$, the morphism $\Sym E^\vee\to \Sym L^\vee$ is flat. Therefore the quasi-isomorphism $\Sym E^\vee\to \O_Y$ gives rise to another quasi-isomorphism
$$\Sym E^\vee\otimes_{\Sym E^\vee}\Sym L^\vee\longiso \O_Y\otimes_{\Sym E^\vee}\Sym L^\vee\,.$$
The left hand side is equal to $\Sym L^\vee$, and the right hand side is equal to $\O_Y\otimes_{\O_M}\Sym \tilde L^\vee$, proving that we have a quasi-isomorphism
$\Sym L^\vee\to \O_Y\otimes_{\O_M}\Sym \tilde L^\vee$. Hence the conclusion follows.
\end{pf}

\begin{prop}\label{prop:TransEmbedQiso}
Any transfer embedding $(U,H,\mu) \to (M,L,\lambda)$ is a quasi-isomorphism of $L_\infty$-bundles.
\end{prop}
\begin{proof}
 We assume that $H\to L$ is obtained by an application of the transfer theorem for bundles of
 curved $L_\infty[1]$-algebras as in Proposition~\ref{transfertheorem}. More precisely, let $(M,L,\delta)$  be a bundle of complexes endowed with the curved $L_\infty[1]$-structure $\lambda$, let $\eta$ be a contraction of $\delta$ and $(M,H,\delta)$ the bundle of complexes onto which $\eta$ contracts $(M,L,\delta)$.  Let $\mu$ be the family of transferred curved $L_\infty[1]$-structures on $(M,H,\delta)$ and $\phi:H\to L$ the inclusion morphism. 

Recall that a transfer embedding is a composition of an inclusion morphism $\phi$ obtained by the transfer theorem with the trivial fibration $(U,H,\mu) \hookrightarrow (M,H,\mu)$. Since the trivial fibration $(U,H,\mu) \hookrightarrow (M,H,\mu)$ is a quasi-isomorphism, it suffices to show that $(\id,\phi):(M,H,\mu) \to (M,L,\lambda)$ is a quasi-isomorphism.

Let $A=\Gamma(M,\Sym L^\vee)$   and $B=\Gamma(M,\Sym H^\vee)$.  Let us denote the derivations induced by the dual maps of $\delta$ and $\lambda_n$, $n \geq 0$, on $A$ by $\delta$ and $q_n$, $n \geq 0$, respectively. Similarly, denote the derivations induced by the dual maps of $\delta$ and $\mu_n$, $n \geq 0$, on $B$ by $\delta$ and $r_n$, $n \geq 0$, respectively.  

\paragraph{Claim.}
The morphism of function algebras
$$\phi^\sharp:(A,\delta+q)\longto (B,\delta+r)\,,$$
associated to the morphism of $L_\infty$-bundles $\phi:(M,H,\delta+\mu)\to (M,L,\delta+\lambda)$ is a quasi-isomorphism.

To prove the claim, refine the grading on $A$ to a double grading
$$A^{k,\ell}=\Gamma(M,\Sym^{-\ell}L^\vee)^{k+\ell}\,.$$
It is contained in the region defined by $\ell\leq0$ and $k\leq0$. Moreover $A$ has only finitely many non-zero terms for each fixed value of $k+\ell$. This will imply that the spectral sequences we construct below are bounded and hence convergent to the expected limit.

Note that $\delta$ and all $q_n$ are bigraded: the degree of $\delta$ is $(1,0)$, and the degree of $q_n$ is $(n,1-n)$. 

If we filter $A$ by $F_k A = A^{\geq k, \bullet}$, then the differential $\delta + q$ preserves this filtration, and we obtain a bounded spectral sequence $E_n^{k, \ell}$ converging to $H^{k+\ell}(A, \delta + q)$:
$$E_n^{k,\ell}\Longrightarrow H^{k+\ell}(A,\delta+q)\,.$$ 
The same construction applies to $(B,\delta+r)$, and we get a convergent spectral sequence
$$\tilde E_n^{k,\ell}\Longrightarrow H^{k+\ell}(B,\delta+r)\,.$$
The morphism of differential graded algebras $\phi^\sharp$ induces a morphism of spectral sequences $E\to \tilde E$, because $\phi$  respects the filtrations introduced above. 

To prove our claim, it suffices to show that $\phi^\sharp$ induces a quasi-isomorphism $E_1\to \tilde E_1$. 
The differential on $E_1$ is induced by $\delta+q_1$ on the cohomology of $E_0$ with respect to $q_0$. Similarly, the differential on $\tilde E_1$ is induced by $\delta+r_1$ on the cohomology of $\tilde  E_0$ with respect to $r_0$. The homomorphism $E_1\to \tilde E_1$ is induced by $\phi_1^\sharp$. 

Recall that from Proposition~\ref{transfertheorem}, we have the deformed projection $\tilde\pi_1:L\to H$.  It induces an algebra morphism $B\to A$, which we will denote by $\tilde\pi_1^\sharp$.  
Also there is the deformed contraction $\tilde\eta:L\to L$ of degree $-1$. 
We extend its dual to a fiberwise derivation $\Sym L^\vee \to \Sym L^\vee$, which we also denote by $\tilde\eta$ by abuse of notations. Denote by $\tilde\eta': \Sym L^\vee \to \Sym L^\vee$ the endomorphism obtained by dividing the derivation $\tilde\eta$ by the weight (and setting $\tilde\eta'(1)=0$).

We will now construct a fiberwise homotopy operator $h:\Sym L^\vee \to \Sym L^\vee$, with the property that 
\begin{equation}\label{heq}
[\delta+q_1, h]=1-\phi_1^\sharp\tilde\pi_1^\sharp\,.
\end{equation}
For this purpose, note that $\Sym L^\vee$ is a bundle of Hopf algebras. Let us denote the coproduct by $\Delta$, and the product by $m$. Define $h$ fiberwise by the formula:
$$h= m \circ w\circ(\phi_1^\sharp\tilde\pi_1^\sharp\otimes\tilde\eta')\circ \Delta\,.$$
Here $w:\Sym L^\vee \otimes \Sym L^\vee \to \Sym L^\vee \otimes \Sym L^\vee$ is the operator that divides an element of bi-weight $(k,\ell)$ by the binomial coefficient $\frac{(k+\ell)!}{k!\ell!}$. 
It follows from a direct verification that \eqref{heq} holds. Hence the map $\phi$ indeed induces a homotopy equivalence on $E_1$.  This concludes the proof.
\end{proof}

\begin{pfqiso}
By Proposition~\ref{prop:IsoToLinearMor} and Theorem~\ref{thm:main}, we may assume that $(M,L,\lambda)\to (N,L',\lambda')$ is a linear trivial fibration of $L_\infty$-bundles.  Then, according to Theorem~\ref{recap}, it suffices to prove that both $(Y,E,\nu)\to(U,H,\mu)$ and $(U,H,\mu)\to (M,L,\lambda)$ induce quasi-isomorphisms on the associated dg algebras of global sections. The former is done in Lemma~\ref{former}.  
The latter follows from Proposition~\ref{prop:TransEmbedQiso}.
\end{pfqiso}

\subsection{On the converse implication}

In this subsection, we give an elementary partial proof of the converse implication: if a morphism of $L_\infty$-bundles induces a quasi-isomorphism on the associated dg algebras of global sections, then it is a weak equivalence.

\begin{prop}\label{prop:QuasiIsoBijMC}
A quasi-isomorphism of $L_\infty$-bundles induces a bijection on classical loci.
\end{prop}
\begin{pf}
Let $\M$ and $\N$ be $L_\infty$-bundles with base manifolds $M$ and $N$, respectively.  We denote by $A$ and $B$ the corresponding cdgas of global sections. Let $(f,\phi): \M \to \N$ be a morphism of $L_\infty$-bundles. Note that the set of points of $M$ is identified with the set of
$\rr$-algebra morphisms $A^0\to \rr$, where $A^0 = C^\infty(M)$. See \cite[Problem 1-C]{MR0440554}.  It follows that the set of
classical points of $\M$ is equal to the set of algebra morphisms
$H^0(A)\to\rr$. Similarly, the set of classical points of $\N$ is the
set of algebra morphisms $H^0(B)\to \rr$.  Since the quasi-isomorphism
$B\to A$ induces an isomorphism $H^0(B)\to H^0(A)$, it follows that $f$
induces a bijection $\pi_0(\M)\to\pi_0(\N)$ on classical loci.
\end{pf}

\begin{prop}\label{prop:QuasiIsoWkEqFib}
If a fibration of quasi-smooth $L_\infty$-bundles induces a quasi-isomorphism of the associated dg algebras of global sections, then it is a weak equivalence.
\end{prop}
\begin{pf}
We are in the quasi-smooth case. So $L=L^1$, and $E=E^1$.  By assumption, the following three-term sequence is exact on the right:
\begin{equation}\label{right}
\begin{split}
\xymatrix{
\Gamma(M,\Lambda^2 L^\vee)\rto^{\contract \lambda} & \Gamma(M,L^\vee)\rto^{\contract\lambda} & C^\infty(M)\\
&\Gamma(N,E^\vee)\rto^{\contract \mu}\uto & C^\infty(N)\uto}
\end{split}
\end{equation}
As a formal consequence, the following diagram is exact on the left:
\begin{equation}\label{left}
\vcenter{\xymatrix{
C^\infty(M)^\ast\rto\dto & \Gamma(M,L^\vee)^\ast\rto\dto & \Gamma(M,\Lambda^2 L^\vee)^\ast\\
C^\infty(N)^\ast\rto & \Gamma(N,E^\vee)^\ast}}
\end{equation}
Here the asterisks denote $\rr$-linear maps to $\rr$. 
Our claim is that for $P\in Z(\lambda)\subset M$, the square
\begin{equation}\label{squr}
\vcenter{\xymatrix{TM|_P\rto\dto &L|_P\dto\\
TN|_{f(P)}\rto & E|_{f(P)}}}\end{equation}
is exact. By our assumption that $(f,\phi)$ is a fibration, the vertical maps are surjective, so we need to prove that (\ref{squr}) induces a bijection on kernels. The point $P$ defines a $C^\infty(M)$-module structure on $\rr$, and the point $f(P)$ defines a $C^\infty(N)$-module structure on $\rr$.  The vector space 
$$TM|_P\subset C^\infty(M)^\ast$$
consists of the linear maps $C^\infty(M)\to\rr$ which are derivations.  The vector space 
$$L|_P\subset\Gamma(M,L^\vee)^\ast$$
consists of the $C^\infty(M)$-linear maps $\Gamma(M,L^\vee)\to\rr$. 

Consider an element $x\in L|_P$, mapping to zero in $E|_{f(P)}$. The element $x$ maps to zero in $\Gamma(M,\Lambda^2L^\vee)^\ast$, because 
the map $L|_P\to \Lambda^2 L|_P$ induced by multiplication by $\lambda$ is zero, as $\lambda$ vanishes at $P$.  Comparing with (\ref{left}), we see that there exists a unique $\rr$-linear map $v:C^\infty(M) \to\rr$, such that 
\begin{items}
\item $v\circ f^\sharp:\Gamma(N)\to\rr$ vanishes;
\item $v\circ(\contract\lambda):\Gamma(L^\vee)\to \rr$ is equal to $x$.
\end{items}
We need to show that $v$ is a derivation, i.e.\  for $g,h\in C^\infty(M)$, we have
$$v(gh)=v(g)h+gv(h)\,.$$
By the right exactness of \eqref{right}, the vector space $C^\infty(M)$ is generated as an $\rr$-vector space by the images of $C^\infty(N)$ and $\Gamma(M,L^\vee)$. Thus, it suffices to consider the following two cases.

{\bf Case 1.} Assume both $g,h$ are pullbacks from $N$. So there exist $\tilde g,\tilde h:N\to\rr$, such that $g=\tilde g\circ f$ and $h=\tilde h\circ f$. We have
$$v(gh)=v((\tilde g\circ f)(\tilde h\circ f))=v((\tilde g\tilde h)\circ f)=0=v(\tilde g\circ f)h+gv(\tilde h\circ f)=v(g)h+gv(h)\,.$$

{\bf Case 2.} Assume   that $h=\alpha\contract\lambda$, where $\alpha\in \Gamma(M,L^\vee)$. We have
$$v(gh)= v(g(\alpha\contract\lambda))=v((g\alpha)\contract\lambda)
=x(g\alpha)=g(P)x(\alpha)\,,$$
and
$$v(g)h+gv(h)=v(g)(\alpha\contract\lambda)(P)+g(P) v(\alpha\contract\lambda)=0+g(P)x(\alpha)\,.$$
This finishes the proof.
\end{pf}

\bibliographystyle{plain}
\bibliography{ref_CatFibObj}

\Addresses

\end{document}